\documentclass[12pt,reqno]{amsart}
\usepackage[a4paper,margin=2.5cm,top=2.5cm,bottom=2.5cm,centering,headheight=3ex,headsep=4ex,vcentering]{geometry}
\usepackage{amssymb,amsfonts,amsmath,amsthm}
\usepackage{orcidlink}
\usepackage{cite}
\usepackage[cal=cm,scr=euler]{mathalfa}
\usepackage{microtype}
\usepackage{mlmodern}
\numberwithin{equation}{section}
\theoremstyle{plain}
\newtheorem{theorem}{Theorem}
\newtheorem{conjecture}{Conjecture}
\newtheorem{lemma}{Lemma}

\newtheorem{corollary}{Corollary}
\theoremstyle{definition}

\theoremstyle{remark}
\newtheorem{remark}{Remark}

\usepackage{xcolor}
\usepackage{hyperref}
\definecolor{ao(english)}{rgb}{0.0, 0.5, 0.0}
\definecolor{pink}{rgb}{1.0, 0.0, 1.0}
\hypersetup{colorlinks=true, linkcolor=ao(english),citecolor=ao(english)}

\title[$q$-Supercongruences Modulo the Third and Fifth Powers of a Cyclotomic Polynomial]{$q$-Supercongruences Modulo the Third and Fifth Powers of a Cyclotomic Polynomial}

\author[Liton Karmakar]{Liton Karmakar\,\orcidlink{0009-0002-9532-1677}}
\address[Liton Karmakar]{Department of Mathematics, National Institute of Technology Silchar, Assam 788010, India}
\email{litonofficial8638@gmail.com}

\author[Arijit Jana]{Arijit Jana\,\orcidlink{0000-0002-7440-3297}}
\address[Arijit Jana]{Department of Mathematics, National Institute of Technology Silchar, Assam 788010, India}
\email{jana94arijit@gmail.com}


\begin{document}

\begin{abstract}
In this paper, we establish several $q$-supercongruences for basic hypergeometric series modulo the third and fifth powers of a cyclotomic polynomial. These results extend several existing supercongruences related to the $\mathrm{(A.2)}$ and $\mathrm{(C.2)}$ supercongruences of Van Hamme. Our proofs employ the creative microscoping method developed by Guo and Zudilin \cite{guo2019q}, together with Watson's ${}_{8}\phi_7$ transformation formula.
\end{abstract}

\maketitle

\section{Introduction and statement of results}

In his celebrated letter to Hardy, Ramanujan \cite{ramanujan} presented several remarkable infinite series for $1/\pi$. Among them are the following identities:
\begin{align} 
    & (\mathrm{A}.1)\qquad
      \sum_{n=0}^{\infty} (-1)^n(4n+1)
       \frac{(\frac{1}{2})_n^5}{n!^5}
      = \frac{2 }{\Gamma(\frac{3}{4})^4}, \label{A1} \\ 
    & (\mathrm{C}.1)\qquad\sum_{n=0}^{\infty} (4n+1) \frac{(\frac{1}{2})_n^4}{n!^4} =\infty, \label{C1} 
\end{align}
where $(z)_n= z (z+1) \cdots(z+n-1)$ denotes the \textit{shifted factorial} and $\Gamma(x)$ denotes the Gamma function. In $1997$, Van Hamme \cite{van2020some} formulated a collection of remarkable conjectural supercongruences that can be viewed as $p$-adic analogues of Ramanujan-type formulas for $1/\pi$. For example, the identities \eqref{A1} and \eqref{C1} give rise to the following notable supercongruences:
\begin{align}
   & (\mathrm{A}.2)\qquad   \displaystyle \sum_{n=0}^{(p-1)/2} (-1)^n (4n+1) \frac{(\frac{1}{2})_n^5}{n!^5} \equiv
\begin{cases}
-\dfrac{p}{\Gamma_p\left(\frac{3}{4}\right)^4}, & p \equiv 1 \pmod{4} \\
0, & p \equiv 3 \pmod{4} ~~~~
\end{cases}
\ (\bmod\ p^3), \label{A2} \\
   & (\mathrm{C}.2)\qquad \sum_{n=0}^{(p-1)/2} (4n+1) \frac{(\frac{1}{2})_n^4}{n!^4} \equiv p \pmod{p^3}, \label{C2} 
\end{align}
where $p$ is an odd prime. In $2008$, McCarthy and Osburn \cite{McCarthyOsburn} provided a proof for the supercongruence \eqref{A2}. Subsequently, Swisher \cite{swisher} extended this result to the modulus $p^5$ for $p \equiv 1 \pmod{4}$ and $p>5.$ On the other hand, Van Hamme himself proved the supercongruence \eqref{C2} by analyzing a sequence of orthogonal polynomials. In $2011$, Long \cite{long1} independently strengthened this result by extending \eqref{C2} to the modulus $p^4$ for primes $p\ge 5$, using certain hypergeometric identities. In particular, Long established the following two supercongruences:
\begin{align}
    & \sum_{n=0}^{(p-1)/2} (4n+1) \frac{(\frac{1}{2})_n^4}{n!^4} \equiv p \pmod{p^4}, \label{LongC2 1} \\
    & \sum_{n=0}^{(p-1)/2} (4n+1) \frac{(\frac{1}{2})_n^6}{n!^6} \equiv p \sum_{n=0}^{(p-1)/2} \frac{(\frac{1}{2})_n^4}{n!^4} \pmod{p^4}. \label{LongC2 2}
\end{align}
Employing the $q$-WZ method, Guo and Wang \cite{Guo2020 Edinburgh} established a $q$-analogue of \eqref{LongC2 1}. More precisely, for any odd integer $m>1$, they proved that

\begin{align} \label{GuoWang qC2 1}
    \sum_{n=0}^{(m-1)/2}
[4n+1]\frac{(q;q^{2})_{n}^{4}}{(q^{2};q^{2})_{n}^{4}}
\equiv q^{(1-m)/2}[m]
+\frac{(m^{2}-1)(1-q)^{2}}{24}\,
q^{(1-m)/2}[m]^{3}
\pmod{[m]\Phi_{m}(q)^{3}}.
\end{align}
By reducing \eqref{GuoWang qC2 1} modulo $[m]\Phi_m(q)^2$, they obtained the following $q$-analogue of the supercongruence \eqref{C2}:
\begin{align} \label{GuoWang qC2 2}
    \sum_{n=0}^{(m-1)/2}
[4n+1]\frac{(q;q^{2})_{n}^{4}}{(q^{2};q^{2})_{n}^{4}}
\equiv q^{(1-m)/2}[m]
\pmod{[m]\Phi_m(q)^{2}}.
\end{align}
Here, as usual, $[m]=(1-q^m)/(1-q)$ denotes the $q$-integer and $\Phi_m(q)$ denotes the $m$-th cyclotomic polynomial. 
%In this context, Guo and Zudilin \cite[Theorem $4.2$ with $a=1$]{guo2019q} established a parametric generalization of \eqref{GuoWang qC2 2} involving an additional parameter $c$. More precisely, for any odd integer $m>1$, they proved that
%\begin{align} \label{GuoZudilin qC2}
%&\sum_{n=0}^{(m-1)/2}[4n+1] \frac{(q;q^2)_{n}^3 (q/c;q^2)_n}{(q^2;q^2)_{n}^3 (cq^2;q^2)_n} c^n \equiv [m] \frac{(c/q)^{\frac{m-1}{2}} (q^2/c;q^2)_{\frac{m-1}{2}}}{(cq^2;q^2)_{\frac{m-1}{2}}}   \pmod{[m]\Phi_m(q)^2}.
%\end{align}
Quite recently, Tang \cite{Tang2023 JMAA} established a new parametric $q$-extension of \eqref{C2}. Specifically, for any odd integer $m>1$ and $0\leq s\leq \frac{m-1}{2}$, Tang proved that
\begin{align} \label{Tang's Extn qC2}
&\sum_{n=s}^{(m-1)/2}[4n+1] \frac{(q;q^2)_{n-s}(q;q^2)_{n+s} (q;q^2)_n^2}{(q^2;q^2)_{n-s}(q^2;q^2)_{n+s} (q^2;q^2)_n^2 } \equiv [m]q^{\frac{1-m}{2}} \pmod{\Phi_m(q)^3}.
\end{align}
Among Van Hamme\rq s remarkable conjectural supercongruences, Van Hamme also proved the following supercongruence:
\begin{align}
    (\mathrm{H}.2)\qquad   \displaystyle \sum_{n=0}^{(p-1)/2} \frac{(\frac{1}{2})_n^3}{n!^3} \equiv
\begin{cases}
-\Gamma_p\left(\frac{1}{4}\right)^4, & p \equiv 1 \pmod{4} \\
0, & p \equiv 3 \pmod{4} ~~~~
\end{cases}
\ (\bmod\ p^2). \label{H2}
\end{align}
Guo and Schlosser \cite{GuoSclssr ConstApprox 21} obtained a $q$-analogue of \eqref{LongC2 2} as follows: for a positive odd integer $m$,
\begin{align} \label{GuoSchlsr q-long}
    \sum_{n=0}^{(m-1)/2}
[4n+1]
\frac{(q;q^2)_n^6}{(q^2;q^2)_n^6}q^n
\equiv
[m]q^{(1-m)/2}
\sum_{n=0}^{(m-1)/2}
\frac{(q;q^2)_n^4}{(q^2;q^2)_n^4}q^{2n}
\pmod{[m]\Phi_m(q)^2}.
\end{align}
From \eqref{A2} and \eqref{H2}, Guo and Schlosser \cite{GuoSclssr ConstApprox 21} deduce that
\begin{align} \label{GuoSchlsr remakr}
    \sum_{n=0}^{(p-1)/2} (-1)^n (4n+1) \frac{(\frac{1}{2})_n^5}{n!^5} \equiv p \sum_{n=0}^{(p-1)/2} \frac{(\frac{1}{2})_n^3}{n!^3} \pmod{p^3},
\end{align}
which was previously observed by Morteson \cite{mortenson}. In the same paper, they also established the following $q$-analogue of \eqref{GuoSchlsr remakr}:
\begin{align} \label{GuoSchlsr q-power 5}
    &\sum_{n=0}^{(m-1)/2}
(-1)^n[4n+1]
\frac{(q;q^2)_n^5}{(q^2;q^2)_n^5}
q^{n^2+n} \equiv
[m]q^{(1-m)/2}
\sum_{n=0}^{(m-1)/2}
\frac{(q;q^2)_n^3}{(q^2;q^2)_n^3}\,
q^{2n}
\pmod{[m]\Phi_m(q)^2}.
\end{align}
%Quite recently, Tang \cite{Tang2023 JMAA} gave a new extension of \eqref{C2} as follows: for $0 \le s \le (p-1)/2$,
%\begin{align} \label{Tang's Extn C2}
%    \sum_{n=s}^{\frac{p-1}{2}} (4n+1) \frac{(\frac{1}{2})_{n-s} (\frac{1}{2})_{n+s} (\frac{1}{2})_{n}^2  }{(n-s)! (n+s)! n!^2} \equiv p \pmod{p^3}.
%\end{align}
%With the help of Jackson\rq s ${}_{6}\phi_5$ summation together with the \lq \lq creative microscopic method\rq \rq, introduced in \cite{guo2019q}, Tang proved \eqref{Tang's Extn C2} by establishing the following $q$-supercongruence: for an odd integer $m>1,$ and $0\le s \le \frac{m-1}{2},$
For completeness, we recall some standard notation from the theory of basic hypergeometric series. Let $z$ and $q$ be complex numbers with $|q|<1$. The \emph{$q$-shifted factorial} is defined by  $$
(z;q)_0:=1,\qquad
(z;q)_n:=\prod_{k=0}^{n-1} (1-zq^k),\qquad n\geq 1.
$$
For convenience, we use the compact notation:
$$(z_1,\ldots,z_r;q)_n
:=(z_1;q)_n\cdots(z_r;q)_n.$$
Recently, Wei \cite{Wei JCTA 21} provided some refinements of the supercongruences $\eqref{LongC2 1} - \eqref{GuoWang qC2 2}$ and $\eqref{GuoSchlsr q-long}-\eqref{GuoSchlsr q-power 5}$.
\par Motivated by these, we aim to establish a parametric $q$-supercongruence that involves the parameters $c,d$, and $s$, as stated in the following theorem.
\begin{theorem} \label{P7 Thm 1}
    Let $m>1$ be an odd integer and $s$ be a non-negative integer with $s\le \frac{m-1}{2}$. Then, modulo $\Phi_m(q)^3,$
\begin{align} \label{P7 Thm 1 eqn}
&\sum_{n=s}^{\frac{m-1}{2}}[4n+1] \frac{(q;q^2)_{n-s}(q;q^2)_{n+s} (q;q^2)_n^2(cq,dq;q^2)_n}{(q^2;q^2)_{n-s}(q^2;q^2)_{n+s} (q^2;q^2)_n^2 (q^2/c, q^2/d; q^2)_n}(cd)^{-n}q^n \notag \\
&\quad \equiv
[m+2s]\frac{(q;q^2)_s^2(cq,dq;q^2)_s (q^2;q^2)_{\frac{m-1}{2}-s} (q;q^2)_{\frac{m-1}{2}+s}}{(q^2;q^2)_s^2(q^2/c,q^2/d;q^2)_s (q;q^2)_{\frac{m-1}{2}-s} (q^2;q^2)_{\frac{m-1}{2}+s}} (cd)^{-s}q^{\frac{1-m}{2}+2s} \notag \\
&\qquad \times\sum_{n=0}^{\frac{m-1}{2}-s} \frac{(q,q/cd;q^2)_n (q^{2s+1};q^2)_n^2}{(q^2;q^2)_n^2(q^{2s+2}/c, q^{2s+2}/d;q^2)_n} q^{2n}.
\end{align}
\end{theorem}
\begin{remark}
Setting $cd=q$ in \eqref{P7 Thm 1 eqn}, Theorem \ref{P7 Thm 1} reduces to Tang's $q$-congruence \eqref{Tang's Extn qC2}, thereby extending $q$-congruence \eqref{GuoWang qC2 2}. 
\end{remark}

We now present some consequences of Theorem~\ref{P7 Thm 1}. Taking $c=d=1$ in Theorem~\ref{P7 Thm 1} yields the following new $q$-extension of \eqref{GuoSchlsr q-long}.
\begin{corollary} \label{P7 cor 1}
    For any odd integer $m>1$, non-negative integer $s$ with $s\le \frac{m-1}{2}$, we have
\begin{align} \label{P7 cor 1 eqn}
&\sum_{n=s}^{\frac{m-1}{2}}[4n+1] \frac{(q;q^2)_{n-s}(q;q^2)_{n+s} (q;q^2)_n^4}{(q^2;q^2)_{n-s}(q^2;q^2)_{n+s} (q^2;q^2)_n^4}q^n \notag \\
&\quad \equiv
[m+2s]\frac{(q;q^2)_s^4 (q^2;q^2)_{\frac{m-1}{2}-s} (q;q^2)_{\frac{m-1}{2}+s}}{(q^2;q^2)_s^4 (q;q^2)_{\frac{m-1}{2}-s} (q^2;q^2)_{\frac{m-1}{2}+s}} q^{\frac{1-m}{2}+2s} \notag \\
&\qquad \times\sum_{n=0}^{\frac{m-1}{2}-s} \frac{(q;q^2)_n^2 (q^{2s+1};q^2)_n^2}{(q^2;q^2)_n^2(q^{2s+2}, q^{2s+2};q^2)_n} q^{2n} \pmod{\Phi_m(q)^3}.
\end{align}
\end{corollary}

Setting $m=p$ in \eqref{P7 cor 1 eqn} and letting $q\to1$, we immediately obtain the following corollary.
\begin{corollary} \label{P7 Cor 2}
    For any odd prime $p$, and $0 \le s \le (p-1)/2$, we have
    \begin{align*} 
&\sum_{n=s}^{\frac{p-1}{2}}(4n+1) \frac{(\frac{1}{2})_{n-s}(\frac{1}{2})_{n+s} (\frac{1}{2})_n^4}{(1)_{n-s}(1)_{n+s} (1)_n^4} \notag \\
&\quad \equiv
(p+2s)\frac{(\frac{1}{2})_s^4 (1)_{\frac{p-1}{2}-s} (\frac{1}{2})_{\frac{p-1}{2}+s}}{(1)_s^4 (\frac{1}{2})_{\frac{p-1}{2}-s} (1)_{\frac{p-1}{2}+s}}  \sum_{n=0}^{\frac{p-1}{2}-s} \frac{(\frac{1}{2})_n^2 (s+\frac{1}{2})_n^2}{(1)_n^2(s+1)_n^2} \pmod{p^3}.
\end{align*}
\end{corollary}

Setting $c=1$ and letting $d\to\infty$ in Theorem~\ref{P7 Thm 1}, we obtain the following new $q$-extension of \eqref{GuoSchlsr q-power 5}.
\begin{corollary} \label{P7 Cor 3}
     For any odd integer $m>1$, non-negative integer $s$ with $s\le \frac{m-1}{2}$, we have
\begin{align}\label{P7 Cor 3 eq}
&\sum_{n=s}^{\frac{m-1}{2}} (-1)^n[4n+1] \frac{(q;q^2)_{n-s}(q;q^2)_{n+s} (q;q^2)_n^3}{(q^2;q^2)_{n-s}(q^2;q^2)_{n+s} (q^2;q^2)_n^3}q^{n^2+n} \notag \\
&\quad \equiv
(-1)^s[m+2s]\frac{(q;q^2)_s^3 (q^2;q^2)_{\frac{m-1}{2}-s} (q;q^2)_{\frac{m-1}{2}+s}}{(q^2;q^2)_s^3 (q;q^2)_{\frac{m-1}{2}-s} (q^2;q^2)_{\frac{m-1}{2}+s}} q^{\frac{1-m}{2}+s(s+2)} \notag \\
&\qquad \times\sum_{n=0}^{\frac{m-1}{2}-s} \frac{(q;q^2)_n (q^{2s+1};q^2)_n^2}{(q^2;q^2)_n^2(q^{2s+2};q^2)_n} q^{2n} \pmod{\Phi_m(q)^3}.
\end{align}
\end{corollary}

Putting $m=p$ in \eqref{P7 Cor 3 eq} and then taking the limit as $q\to1$, we arrive at the following corollary.
\begin{corollary} \label{P7 Cor 4}
    For any odd prime $p$, and $0 \le s \le (p-1)/2$, we have
    \begin{align} \label{P8 Cor 4 eqn}
&\sum_{n=s}^{\frac{p-1}{2}} (-1)^n (4n+1) \frac{(\frac{1}{2})_{n-s}(\frac{1}{2})_{n+s} (\frac{1}{2})_n^3}{(1)_{n-s}(1)_{n+s} (1)_n^3} \notag \\
&\quad \equiv
(-1)^s(p+2s)\frac{(\frac{1}{2})_s^3 (1)_{\frac{p-1}{2}-s} (\frac{1}{2})_{\frac{p-1}{2}+s}}{(1)_s^3 (\frac{1}{2})_{\frac{p-1}{2}-s} (1)_{\frac{p-1}{2}+s}}  \sum_{n=0}^{\frac{p-1}{2}-s} \frac{(\frac{1}{2})_n (s+\frac{1}{2})_n^2}{(1)_n^2(s+1)_n} \pmod{p^3}.
\end{align}
\end{corollary}
\begin{remark}
    Putting $s=0$ in the congruence \eqref{P8 Cor 4 eqn}, we achieve \eqref{GuoSchlsr remakr}.
\end{remark}
Moreover, setting $c=-1$ and $d=1$ in Theorem~\ref{P7 Thm 1}, we obtain the following  result.

\begin{corollary}
    For any odd integer $m>1$, non-negative integer $s$ with $s\le \frac{m-1}{2}$, we have
\begin{align*}
    &\sum_{n=s}^{\frac{m-1}{2}} (-1)^n[4n+1] \frac{(q;q^2)_{n-s}(q;q^2)_{n+s} (q;q^2)_n^2(q^2;q^4)_n}{(q^2;q^2)_{n-s}(q^2;q^2)_{n+s} (q^2;q^2)_n^2 (q^4; q^4)_n}q^n \notag \\
&\quad \equiv
(-1)^s[m+2s]\frac{(q;q^2)_s^2(q^2;q^4)_s (q^2;q^2)_{\frac{m-1}{2}-s} (q;q^2)_{\frac{m-1}{2}+s}}{(q^2;q^2)_s^2(q^4;q^4)_s (q;q^2)_{\frac{m-1}{2}-s} (q^2;q^2)_{\frac{m-1}{2}+s}} q^{\frac{1-m}{2}+2s} \notag \\
&\qquad \times\sum_{n=0}^{\frac{m-1}{2}-s} \frac{(q^2;q^4)_n (q^{2s+1};q^2)_n^2}{(q^2;q^2)_n^2(q^{4s+4};q^4)_n} q^{2n} \pmod{\Phi_m(q)^3}.
\end{align*}
\end{corollary}
Setting $c=q^{-2}$ and letting $d\to\infty$ in Theorem~\ref{P7 Thm 1}, we obtain the following corollary.
\begin{corollary}
     For any odd integer $m>1$, non-negative integer $s$ with $s\le \frac{m-1}{2}$, we have
     \begin{align*} 
&\sum_{n=s}^{\frac{m-1}{2}} (-1)^n[4n+1] \frac{(q;q^2)_{n-s}(q;q^2)_{n+s} (q;q^2)_n^2(q^{-1};q^2)_n }{(q^2;q^2)_{n-s}(q^2;q^2)_{n+s} (q^2;q^2)_n^2 (q^4; q^2)_n}q^{n^2+3n} \notag \\
&\quad \equiv
(-1)^s[m+2s]\frac{(q;q^2)_s^2(q^{-1};q^2)_s (q^2;q^2)_{\frac{m-1}{2}-s} (q;q^2)_{\frac{m-1}{2}+s}}{(q^2;q^2)_s^2(q^4;q^2)_s (q;q^2)_{\frac{m-1}{2}-s} (q^2;q^2)_{\frac{m-1}{2}+s}} q^{\frac{1-m}{2}+s(s+4)} \notag \\
&\qquad \times\sum_{n=0}^{\frac{m-1}{2}-s} \frac{(q;q^2)_n (q^{2s+1};q^2)_n^2}{(q^2;q^2)_n^2(q^{2s+4};q^2)_n} q^{2n} \pmod{\Phi_m(q)^3}.
\end{align*}
\end{corollary}
Similarly, setting $c=q$ and letting $d\to\infty$ in Theorem~\ref{P7 Thm 1}, we obtain the following corollary.
\begin{corollary}
     For any odd integer $m>1$, non-negative integer $s$ with $s\le \frac{m-1}{2}$, we have
     \begin{align*}
         &\sum_{n=s}^{\frac{m-1}{2}} (-1)^n[4n+1] \frac{(q;q^2)_{n-s}(q;q^2)_{n+s} (q;q^2)_n}{(q^2;q^2)_{n-s}(q^2;q^2)_{n+s} (q^2;q^2)_n}q^{n^2} \notag \\
&\quad \equiv
(-1)^s[m+2s]\frac{(q;q^2)_s (q^2;q^2)_{\frac{m-1}{2}-s} (q;q^2)_{\frac{m-1}{2}+s}}{(q^2;q^2)_s (q;q^2)_{\frac{m-1}{2}-s} (q^2;q^2)_{\frac{m-1}{2}+s}} q^{\frac{1-m}{2}+s(s+1)} \notag \\
&\qquad \times\sum_{n=0}^{\frac{m-1}{2}-s} \frac{(q;q^2)_n (q^{2s+1};q^2)_n^2}{(q^2;q^2)_n^2(q^{2s+1};q^2)_n} q^{2n} \pmod{\Phi_m(q)^3}.
     \end{align*}
\end{corollary}

In addition, using Theorem~\ref{P7 Thm 1}, we establish a new $q$-supercongruence modulo $\Phi_m(q)^5$, which extends the previous results.
\begin{theorem} \label{P7 Thm 2}
    Let $m>1$ be an odd integer and $s$ be a non-negative integer with $s\le \frac{m-1}{2}$. Then, modulo $\Phi_m(q)^5,$
\begin{align} \label{P7 Thm 2 eqn}
&\sum_{n=s}^{\frac{m-1}{2}}[4n+1] \frac{(q;q^2)_{n-s}(q;q^2)_{n+s} (q;q^2)_n^4}{(q^2;q^2)_{n-s}(q^2;q^2)_{n+s} (q^2;q^2)_n^4}q^n \notag \\
&\quad \equiv
[m+2s]\frac{(q;q^2)_s^4 (q^2;q^2)_{\frac{m-1}{2}-s} (q;q^2)_{\frac{m-1}{2}+s}}{(q^2;q^2)_s^4(q;q^2)_{\frac{m-1}{2}-s} (q^2;q^2)_{\frac{m-1}{2}+s}} q^{\frac{1-m}{2}+2s} \notag \\
&\qquad \times \sum_{n=0}^{\frac{m-1}{2}-s} \frac{(q;q^2)_n^2 (q^{2s+1};q^2)_n^2}{(q^2;q^2)_n^2(q^{2s+2};q^2)_n^2} q^{2n} \left\{ 1 - (2-q^m)[m]^2 \sum_{i=1}^{n} \Theta_q(s,i)\right\},
\end{align}
where 
\begin{align*}
    \Theta_q(s,i) = \frac{q^{2(s+i)-1}}{[2(s+i)-1]^2} + \frac{q^{2(s+i)}}{[2(s+i)]^2}.
\end{align*}
\end{theorem}
Similarly, setting $m=p$ in \eqref{P7 Thm 2 eqn} and letting $q\to1$, we obtain the following supercongruence.
\begin{corollary}
    For any odd prime $p$, and $0 \le s \le (p-1)/2$, we have
\begin{align} \label{cor8 eqn}
\sum_{n=s}^{\frac{p-1}{2}}(4n+1) & \frac{(\frac{1}{2})_{n-s}(\frac{1}{2})_{n+s} (\frac{1}{2})_n^4}{(1)_{n-s}(1)_{n+s} (1)_n^4} \quad \equiv
(p+2s)\frac{(\frac{1}{2})_s^4 (1)_{\frac{p-1}{2}-s} (\frac{1}{2})_{\frac{p-1}{2}+s}}{(1)_s^4(\frac{1}{2})_{\frac{p-1}{2}-s} (1)_{\frac{p-1}{2}+s}}  \notag \\
&\qquad \times \sum_{n=0}^{\frac{p-1}{2}-s} \frac{(\frac{1}{2})_n^2 (s+\frac{1}{2})_n^2}{(1)_n^2(s+1)_n^2} \left\{ 1 - p^2 H_{2(s+n)}^{(2)} + p^2 H_{2s}^{(2)}  \right\} \pmod{p^5},
\end{align}
where the harmonic numbers of $k$-order are defined by
\begin{align*}
    H_{m}^{(k)}:= \sum_{n=1}^{m} \frac{1}{n^k}~~ \text{with}~~ k,m \ge 1.
\end{align*}
\end{corollary}
\begin{remark}
 The case $s=0$ in the congruence \eqref{cor8 eqn} extends the supercongruence \eqref{LongC2 2} of Long to the modulus $p^5$.
\end{remark}
 
It is worth noting that several interesting parametric extensions of Van Hamme's Ramanujan-type supercongruences can be found in the literature; see, for example, \cite{Guo-qB2C2, Guo25, HeWang2023_D2, Liton2025 Ramanujan, KJ 2026 RJ, P2 Liton, LiWang26, Tang2023 JMAA}.

The rest of the paper is arranged as follows. In Section \ref{Sec2}, we first define the basic hypergeometric series and then recall the Watson\rq s ${}_{8}\phi_7$ transformation to prove our main results. Section \ref{Sec 3} is devoted to the proofs of Theorems \ref{P7 Thm 1} and \ref{P7 Thm 2}. Finally, we conclude the paper by proposing two open problems in Section \ref{Sec4} for interested readers.

\section{Preliminaries} \label{Sec2}
In this section, we begin by recalling the definition of the basic hypergeometric series, whose summation and transformation formulas are frequently used in the proofs of $q$-supercongruences. Following Gasper and Rahman \cite{gasper}, the basic hypergeometric ${}_{r+1}\phi_r$ series is defined by
\begin{align*}
{}_{r+1}\phi_{r}
\left[
\begin{matrix}
a_1, a_2, \ldots, a_{r+1} \\
b_1, \ldots, b_r
\end{matrix}
; q, z
\right]
:= \sum_{n=0}^{\infty}
\frac{(a_1, a_2, \ldots, a_{r+1}; q)_n}
{(q, b_1, \ldots, b_r; q)_n}
z^n.
\end{align*}
Basic hypergeometric series satisfy a variety of summation and transformation formulas. In this paper, we will make use of the following identity in the proofs of our main theorems.

\begin{lemma}[Watson\rq s ${}_{8}\phi_7$ transformation, Appendix (III.18), \cite{gasper}] 
\begin{align} \label{Watson}
&{}_8\phi_7\left[
\begin{matrix}
a, & qa^{\frac{1}{2}}, & -qa^{\frac{1}{2}}, & b, & c, & d, & e, & q^{-N} \\
& a^{\frac{1}{2}}, & -a^{\frac{1}{2}}, & aq/b, & aq/c, & aq/d, & aq/e, & aq^{N+1}
\end{matrix}
~~; q, \dfrac{a^2q^{N+2}}{bcde}
\right] \notag \\
&\qquad =
\frac{(aq,aq/de;q)_N}{(aq/d,aq/e;q)_N}{}_4\phi_3\!\left[
\begin{matrix}
aq/bc, & d, & e, & q^{-N}\\
& aq/b, & aq/c, & deq^{-N}/a
\end{matrix}
;\,q,\;q
\right].
\end{align}
\end{lemma}
In addition, we require the following two lemmas.
\begin{lemma} [Lemma $2.1$, \cite{GuoSchlosser Isr J}]  \label{P7 Lemma 1}
    Let $d$, $J$ and $m$ be positive integers with $J\le m-1$, and $x$ an integer satisfying
$dJ\equiv -x \pmod m$. Then, for $0\le n\le J$, we have
\begin{align} \label{P7 lemma 1 Eqn}
    \frac{(aq^x;q^d)_{J-n}}{(q^d/a;q^d)_{J-n}}
\equiv
(-a)^{J-2n} q^{J(dJ-d+2x)/2+(d-x)n}
\frac{(aq^x;q^d)_n}{(q^d/a;q^d)_n}
\pmod{\Phi_m(q)}.
\end{align}
\end{lemma}

\begin{lemma}  \label{P7 Lemma 2}
    Let $d$, $J$ and $m$ be positive integers with $J\le m-1$, and $x$ an integer satisfying
$dJ\equiv -x \pmod m$. Then, for a non-negative integer $s$ such that $s\le n\le J-s$, we have

\begin{align} \label{P7 lemma 2 Eqn 1}
    \frac{(aq^x;q^d)_{J-n-s}}{(q^d/a;q^d)_{J-n+s}}
\equiv
(-a)^{J-2n} q^{J(dJ-d+2x)/2+(d-x)n + xs + 2dns}
\frac{(aq^x;q^d)_{n-s}}{(q^d/a;q^d)_{n+s}}
\pmod{\Phi_m(q)},
\end{align}
and
\begin{align} \label{P7 lemma 2 Eqn 2}
    \frac{(aq^x;q^d)_{J-n+s}}{(q^d/a;q^d)_{J-n-s}}
\equiv
(-a)^{J-2n} q^{J(dJ-d+2x)/2+(d-x)n - xs - 2dns}
\frac{(aq^x;q^d)_{n+s}}{(q^d/a;q^d)_{n-s}}
\pmod{\Phi_m(q)}.
\end{align}
\end{lemma}
\begin{proof}
    Since $q^{dJ+x} \equiv q^m \equiv 1 \pmod{ \Phi_m(q)}$. Using Lemma \ref{P7 Lemma 1}, it is easy to find that
    \begin{align*} 
&\displaystyle \frac{(aq^x;q^d)_{J-n-s}}
{(q^d/a;q^d)_{J-n+s}} \notag \\
&=
\frac{ \frac{(aq^x;q^d)_{J-n-s}}{(q^d/a;q^d)_{J-n-s}}          }
{(1- \frac{q^{d+dJ-dn-ds}}{a})
 (1- \frac{q^{2d+dJ-dn-ds}}{a})
 \cdots
 (1- \frac{q^{dJ-dn+ds}}{a})} \notag \\
&\equiv
\frac{
(-a)^{J-2n-2s }
q^{J(dJ-d+2x)/2+(d-x)(n+s)} \frac{(aq^x;q^d)_{n+s}}{(q^d/a;q^d)_{n+s}} }
{(1- \frac{q^{d-x-dn-ds}}{a})
 (1- \frac{q^{2d-x-dn-ds}}{a})
 \cdots
 (1- \frac{q^{-x-dn+ds}}{a})} \notag \\
&\equiv
(-a)^{J-2n}
q^{J(dJ-d+2x)/2+(d-x)n + xs + 2dns}
\frac{(aq^x;q^d)_{n-s}}
{(q^d/a;q^d)_{n+s}}
\pmod{\Phi_m(q)} .
\end{align*} 
This proves \eqref{P7 lemma 2 Eqn 1}. The proof of \eqref{P7 lemma 2 Eqn 2} follows a similar approach.
\end{proof}

\section{Proof of Theorems \ref{P7 Thm 1} and \ref{P7 Thm 2}} \label{Sec 3}
Prior to proving Theorem \ref{P7 Thm 1}, we first derive the following generalized $q$-congruence.
\begin{theorem} \label{P7 theorem 3}
    Let $m>1$ be an odd integer and $s$ be a non-negative integer with $s\le \frac{m-1}{2}$. Then, modulo $\Phi_m(q)(1-aq^m)(a-q^m),$
\begin{align} \label{P7 theorem 3 eqn}
&\sum_{n=s}^{\frac{m-1}{2}}[4n+1] \frac{(q;q^2)_{n-s}(q;q^2)_{n+s} (aq,q/a,cq,dq;q^2)_n}{(q^2;q^2)_{n-s}(q^2;q^2)_{n+s} (aq^2, q^2/a, q^2/c, q^2/d; q^2)_n}(cd)^{-n}q^n \notag \\
&\quad \equiv
[m+2s]\frac{(aq,q/a,cq,dq;q^2)_s (q^2;q^2)_{\frac{m-1}{2}-s} (q;q^2)_{\frac{m-1}{2}+s}}{(aq^{2}, q^{2}/a,q^2/c,q^2/d;q^2)_s (q;q^2)_{\frac{m-1}{2}-s} (q^2;q^2)_{\frac{m-1}{2}+s}} (cd)^{-s}q^{\frac{1-m}{2}+2s} \notag \\
&\qquad \times\sum_{n=0}^{\frac{m-1}{2}-s} \frac{(q,q/cd,aq^{2s+1},q^{2s+1}/a;q^2)_n}{(q^2,q^2,q^{2s+2}/c, q^{2s+2}/d;q^2)_n} q^{2n} .
\end{align}
\end{theorem}
\begin{proof}
    For $a=q^m$ or $a=q^{-m},$ the left-hand side of \eqref{P7 theorem 3 eqn} is equal to
    \begin{align} \label{P7 theorem 3 eqn 2}
        & \sum_{n=s}^{\frac{m-1}{2}} [4n+1]
\frac{(q;q^2)_{n-s}(q;q^2)_{n+s} (q^{1+m},q^{1-m},cq,dq;q^2)_n}
{(q^2;q^2)_{n-s}(q^2;q^2)_{n+s} (q^{2+m}, q^{2-m}, q^2/c, q^2/d; q^2)_n} (cd)^{-n}q^n \notag \\
& = \sum_{n=0}^{\frac{m-1}{2}-s} [4n+4s+1]
\frac{(q;q^2)_{n}(q;q^2)_{n+2s} (q^{1+m},q^{1-m},cq,dq;q^2)_{n+s}}
{(q^2;q^2)_{n}(q^2;q^2)_{n+2s} (q^{2+m}, q^{2-m}, q^2/c, q^2/d; q^2)_{n+s}}
(cd)^{-n-s}q^{n+s} \notag \\
& = [4s+1](cd)^{-s}q^s \frac{(q;q^2)_{2s} (q^{1+m},q^{1-m},cq,dq;q^2)_s}{(q^2;q^2)_{2s} (q^{2+m}, q^{2-m},q^2/c,q^2/d;q^2)_s}
\notag\\
&\quad\times
{}_8\phi_7\left[
\begin{matrix}
q^{4s+1}, & q^{2s+\frac{5}{2}}, & -q^{2s+\frac{5}{2}},
& cq^{2s+1}, & dq^{2s+1}, & q, & q^{2s+1+m}, & q^{2s+1-m} \\
q^{2s+\frac{1}{2}}, & -q^{2s+\frac{1}{2}},
& q^{2s+2}/c, & q^{2s+2}/d, & q^{4s+2}, & q^{2s+2-m}, & q^{2s+2+m}
\end{matrix}
; q^2, \frac{q}{cd}
\right].
    \end{align}
Making use of the parameter substitutions $q \mapsto q^2$, $a=q^{4s+1}$, $b=cq^{2s+1}$, $c=dq^{2s+1}$, $d=q$, $e=q^{2s+1+m}$ and $N \mapsto \frac{m-1}{2}-s$ in Watson\rq s ${}_8\phi_7$ transformation formula \eqref{Watson}, the right-hand side of \eqref{P7 theorem 3 eqn 2} can be rewritten as
\begin{align} \label{P7 theorem 3 eqn 3}
  &  [4s+1](cd)^{-s}q^s \frac{(q;q^2)_{2s} (q^{1+m},q^{1-m},cq,dq;q^2)_s}{(q^2;q^2)_{2s} (q^{2+m}, q^{2-m},q^2/c,q^2/d;q^2)_s} \notag \\
&\quad \times \frac{(q^{4s+3},q^{2s+1-m};q^2)_{\frac{m-1}{2}-s}}{(q^{4s+2},q^{2s+2-m};q^2)_{\frac{m-1}{2}-s}} \sum_{n=0}^{\frac{m-1}{2}-s} \frac{(q/cd,q,q^{2s+1+m},q^{2s+1-m};q^2)_n}{(q^2,q^2,q^{2s+2}/c, q^{2s+2}/d;q^2)_n} q^{2n} \notag \\
& = [4s+1](cd)^{-s}q^s \frac{(q;q^2)_{2s} (q^{1+m},q^{1-m},cq,dq;q^2)_s}{(q^2;q^2)_{2s} (q^{2+m}, q^{2-m},q^2/c,q^2/d;q^2)_s} \frac{(q^2,q^{4s+3};q^2)_{\frac{m-1}{2}-s}}{(q,q^{4s+2};q^2)_{\frac{m-1}{2}-s}} q^{\frac{1-m}{2}+s} \notag \\
&\quad \times \sum_{n=0}^{\frac{m-1}{2}-s} \frac{(q,q/cd,q^{2s+1+m},q^{2s+1-m};q^2)_n}{(q^2,q^2,q^{2s+2}/c, q^{2s+2}/d;q^2)_n} q^{2n} \notag \\
& = [m+2s] \frac{(q^{1+m},q^{1-m},cq,dq;q^2)_s (q^2;q^2)_{\frac{m-1}{2}-s} (q;q^2)_{\frac{m-1}{2}+s}}{(q^{2+m}, q^{2-m},q^2/c,q^2/d;q^2)_s (q;q^2)_{\frac{m-1}{2}-s} (q^2;q^2)_{\frac{m-1}{2}+s}} (cd)^{-s}q^{\frac{1-m}{2}+2s} \notag \\
&\quad \times \sum_{n=0}^{\frac{m-1}{2}-s} \frac{(q,q/cd,q^{2s+1+m},q^{2s+1-m};q^2)_n}{(q^2,q^2,q^{2s+2}/c, q^{2s+2}/d;q^2)_n} q^{2n}.
\end{align}
Therefore, the $q$-supercongruence \eqref{P7 theorem 3 eqn} is true modulo $a - q^{m}$ and $1-aq^m$. By setting $d=2, x=1$ and $J=\frac{m-1}{2},$ and utilising $q$-congruences \eqref{P7 lemma 1 Eqn}, \eqref{P7 lemma 2 Eqn 1} and \eqref{P7 lemma 2 Eqn 2} displayed in Lemmas \ref{P7 Lemma 1} and \ref{P7 Lemma 2}, it is routine to check that, for $s \le n \le M-s,$
\begin{align*}
    & [4(M-n)+1]
\frac{(q;q^2)_{M-n-s}(q;q^2)_{M-n+s} (aq,q/a,cq,dq;q^2)_{M-n}}
{(q^2;q^2)_{M-n-s}(q^2;q^2)_{M-n+s} (aq^2, q^2/a, q^2/c, q^2/d; q^2)_{M-n}}
(cd)^{n-M}q^{M-n} \\
&\equiv - [4n+1]
\frac{(q;q^2)_{n-s}(q;q^2)_{n+s} (aq,q/a,cq,dq;q^2)_n}
{(q^2;q^2)_{n-s}(q^2;q^2)_{n+s} (aq^2, q^2/a, q^2/c, q^2/d; q^2)_n}
(cd)^{-n}q^n \pmod{\Phi_m(q)}.
\end{align*}
This means that the partial sum $\sum_{n=s}^{\frac{m-1}{2}-s} \mathcal{A}_q(n)$ is congruent to $0$ modulo $\Phi_m(q),$ where $\mathcal{A}_q(n)$ denote the $n$-th term on the left-hand side of \eqref{P7 theorem 3 eqn}, that is,
\begin{align*}
    \mathcal{A}_{q}(n)= [4n+1]
\frac{(q;q^2)_{n-s}(q;q^2)_{n+s} (aq,q/a,cq,dq;q^2)_n}
{(q^2;q^2)_{n-s}(q^2;q^2)_{n+s} (aq^2, q^2/a, q^2/c, q^2/d; q^2)_n}
(cd)^{-n}q^n.
\end{align*}
Moreover, for $n$ lying in the interval $  \left( \frac{m-1}{2}-s,  \frac{m-1}{2}\right]$, one can see that $\mathcal{A}_q(n)$ is congruent to $0$ modulo $\Phi_m(q)$ because of the $q$-shifted factorial $(q;q^2)_{n+s}$ present in the numerator of $\mathcal{A}_q(n)$, while the denominator of $\mathcal{A}_q(n)$ is relatively prime to $\Phi_m(q)$ since $m$ is odd. Note that
\begin{align} \label{relation}
    [m+2s] (q;q^2)_{\frac{m-1}{2}+s} = [m](q;q^2)_{\frac{m-1}{2}} (q^{m+2};q^2)_s \equiv 0 \pmod{\Phi_m(q)},
\end{align}
for $m>1$. The relation \eqref{relation} ensures that the summand on the right-hand side of \eqref{P7 theorem 3 eqn} is congruent to $0$ modulo $\Phi_m(q)$ because its denominator is relatively prime to $\Phi_m(q)$. This implies that the $q$-congruence \eqref{P7 theorem 3 eqn} is true modulo $\Phi_m(q)$. Since the polynomials $\Phi_m(q)$, $(1-aq^m)$, and $(a-q^m)$ are pairwise coprime. Thus, by the Chinese remainder theorem for coprime polynomials, the $q$-congruence \eqref{P7 theorem 3 eqn} holds modulo $\Phi_m(q)(1-aq^m)(a-q^m)$. This completes the proof.
\end{proof}

\begin{proof} [Proof of Theorem \ref{P7 Thm 1}]
     Letting $a=1$ in the $q$-congruence \eqref{P7 theorem 3 eqn}, we obtain \eqref{P7 Thm 1 eqn} is true modulo $\Phi_m(q)^3.$ This immediately complete the proof of Theorem \ref{P7 Thm 1}.
\end{proof}
We are now ready to prove Theorem \ref{P7 Thm 2}, for which we first establish the following result.
\begin{theorem} \label{P7 theorem 4}
    Let $m>1$ be an odd integer and $s$ be a non-negative integer with $s\le \frac{m-1}{2}$. Then, modulo $\Phi_m(q)(1-aq^m)(a-q^m)(1-bq^m)(b-q^m),$
\begin{align} \label{P7 theorem 4 eqn}
&\sum_{n=s}^{\frac{m-1}{2}}[4n+1] \frac{(q;q^2)_{n-s}(q;q^2)_{n+s} (aq,q/a,bq,q/b;q^2)_n}{(q^2;q^2)_{n-s}(q^2;q^2)_{n+s} (aq^2, q^2/a, q^2/b, bq^2; q^2)_n}q^n \notag \\
& \equiv [m+2s] \frac{(aq,q/a,bq,q/b;q^2)_s (q^2;q^2)_{\frac{m-1}{2}-s} (q;q^2)_{\frac{m-1}{2}+s}}{(aq^{2}, q^{2}/a,q^2/b,bq^2;q^2)_s (q;q^2)_{\frac{m-1}{2}-s} (q^2;q^2)_{\frac{m-1}{2}+s}} q^{\frac{1-m}{2}+2s} \notag \\
    & \quad \times \left\{ \Omega_1(a,b,q,m) \sum_{n=0}^{\frac{m-1}{2}-s} \frac{(q;q^2)_n^2(aq^{2s+1},q^{2s+1}/a;q^2)_n}{(q^2;q^2)_n^2(q^{2s+2}/b, bq^{2s+2};q^2)_n} q^{2n}  \right. \notag \\ &\left.  \qquad + ~\Omega_2(a,b,q,m) \sum_{n=0}^{\frac{m-1}{2}-s} \frac{(q;q^2)_n^2(bq^{2s+1},q^{2s+1}/b;q^2)_n}{(q^2;q^2)_n^2(q^{2s+2}/a, aq^{2s+2};q^2)_n} q^{2n} \right\},
\end{align}
where \begin{align*}
    & \Omega_1(a,b,q,m) = \frac{(1-bq^m)(b-q^m)(-1-a^2+aq^m)}{(a-b)(1-ab)}, \\
    & \Omega_2(a,b,q,m) = \frac{(1-aq^m)(a-q^m)(-1-b^2+bq^m)}{(b-a)(1-ab)}. 
\end{align*}
\end{theorem}
\begin{proof}
    Substituting $c=b$ and $d=\frac{1}{b}$ in $q$-congruence \eqref{P7 theorem 3 eqn 3}, we obtain
    \begin{align*}
        &\sum_{n=s}^{\frac{m-1}{2}}[4n+1] \frac{(q;q^2)_{n-s}(q;q^2)_{n+s} (q^{1+m},q^{1-m},bq,q/b;q^2)_n}{(q^2;q^2)_{n-s}(q^2;q^2)_{n+s} (q^{2+m}, q^{2-m}, q^2/b, bq^2; q^2)_n}q^n  \\
        & = [m+2s] \frac{(q^{1+m},q^{1-m},bq,q/b;q^2)_s (q^2;q^2)_{\frac{m-1}{2}-s} (q;q^2)_{\frac{m-1}{2}+s}}{(q^{2+m}, q^{2-m},q^2/b,bq^2;q^2)_s (q;q^2)_{\frac{m-1}{2}-s} (q^2;q^2)_{\frac{m-1}{2}+s}} q^{\frac{1-m}{2}+2s} \\
&\quad \times \sum_{n=0}^{\frac{m-1}{2}-s} \frac{(q;q^2)_n^2(q^{2s+1+m},q^{2s+1-m};q^2)_n}{(q^2;q^2)_n^2(q^{2s+2}/b, bq^{2s+2};q^2)_n} q^{2n}.
    \end{align*}
Thus, under modulo $(1-aq^m)(a-q^m)$, we obtain
\begin{align} \label{P7 theorem 4 eqn 2}
    &\sum_{n=s}^{\frac{m-1}{2}}[4n+1] \frac{(q;q^2)_{n-s}(q;q^2)_{n+s} (aq,q/a,bq,q/b;q^2)_n}{(q^2;q^2)_{n-s}(q^2;q^2)_{n+s} (aq^{2}, q^{2}/a, q^2/b, bq^2; q^2)_n}q^n \notag \\
        & \equiv [m+2s] \frac{(aq,q/a,bq,q/b;q^2)_s (q^2;q^2)_{\frac{m-1}{2}-s} (q;q^2)_{\frac{m-1}{2}+s}}{(aq^{2}, q^{2}/a,q^2/b,bq^2;q^2)_s (q;q^2)_{\frac{m-1}{2}-s} (q^2;q^2)_{\frac{m-1}{2}+s}} q^{\frac{1-m}{2}+2s} \notag \\
&\quad \times \sum_{n=0}^{\frac{m-1}{2}-s} \frac{(q;q^2)_n^2(aq^{2s+1},q^{2s+1}/a;q^2)_n}{(q^2;q^2)_n^2(q^{2s+2}/b, bq^{2s+2};q^2)_n} q^{2n} \pmod{(1-aq^m)(a-q^m)}.
\end{align}
Replacing $a$ with $b$ in the above congruence \eqref{P7 theorem 4 eqn 2}, we obtain
\begin{align*} \label{P7 theorem 4 eqn 3}
    &\sum_{n=s}^{\frac{m-1}{2}}[4n+1] \frac{(q;q^2)_{n-s}(q;q^2)_{n+s} (aq,q/a,bq,q/b;q^2)_n}{(q^2;q^2)_{n-s}(q^2;q^2)_{n+s} (aq^{2}, q^{2}/a, q^2/b, bq^2; q^2)_n}q^n \notag \\
    & \equiv [m+2s] \frac{(aq,q/a,bq,q/b;q^2)_s (q^2;q^2)_{\frac{m-1}{2}-s} (q;q^2)_{\frac{m-1}{2}+s}}{(aq^{2}, q^{2}/a,q^2/b,bq^2;q^2)_s (q;q^2)_{\frac{m-1}{2}-s} (q^2;q^2)_{\frac{m-1}{2}+s}} q^{\frac{1-m}{2}+2s} \notag \\
&\quad \times \sum_{n=0}^{\frac{m-1}{2}-s} \frac{(q;q^2)_n^2(bq^{2s+1},q^{2s+1}/b;q^2)_n}{(q^2;q^2)_n^2(q^{2s+2}/a, aq^{2s+2};q^2)_n} q^{2n} \pmod{(1-bq^m)(b-q^m)}.
\end{align*}
Following the same approach used in the proof of Theorem \ref{P7 theorem 3}, it is easy to check that $\sum_{n=s}^{\frac{m-1}{2}}\mathcal{B}_q(n)$ is congruent to $0$, where $\mathcal{B}_q(n)$ denotes the $n$-th summand on the left-hand side of \eqref{P7 theorem 4 eqn}. Further, the summand on the right-hand side of \eqref{P7 theorem 4 eqn} is congruent to $0$ due to the relation \eqref{relation}. Since the polynomials $\Phi_m(q),$ $(1-aq^m)(a-q^m)$, and $(1-bq^m)(b-q^m)$ are pairwise coprime. Therefore, by applying the Chinese remainder theorem for coprime polynomials, we can easily evaluate the remainder of the left-hand side of \eqref{P7 theorem 4 eqn} modulo $\Phi_m(q)(1-aq^m)(a-q^m)(1-bq^m)(b-q^m)$. We can employ the following relations:
\begin{align*}
    & \Omega_1(a,b,q,m) \equiv 1 \pmod{(1-aq^m)(a-q^m)}, \\
    & \Omega_2(a,b,q,m) \equiv 1 \pmod{(1-bq^m)(b-q^m)},
\end{align*}
which allows us to conclude that modulo $\Phi_m(q)(1-aq^m)(a-q^m)(1-bq^m)(b-q^m)$, we have
\begin{align*}
    &\sum_{n=s}^{\frac{m-1}{2}}[4n+1] \frac{(q;q^2)_{n-s}(q;q^2)_{n+s} (aq,q/a,bq,q/b;q^2)_n}{(q^2;q^2)_{n-s}(q^2;q^2)_{n+s} (aq^2, q^2/a, q^2/b, bq^2; q^2)_n}q^n \\
    & \equiv [m+2s] \frac{(aq,q/a,bq,q/b;q^2)_s (q^2;q^2)_{\frac{m-1}{2}-s} (q;q^2)_{\frac{m-1}{2}+s}}{(aq^{2}, q^{2}/a,q^2/b,bq^2;q^2)_s (q;q^2)_{\frac{m-1}{2}-s} (q^2;q^2)_{\frac{m-1}{2}+s}} q^{\frac{1-m}{2}+2s} \\
    & \quad \times \left\{ \Omega_1(a,b,q,m) \sum_{n=0}^{\frac{m-1}{2}-s} \frac{(q;q^2)_n^2(aq^{2s+1},q^{2s+1}/a;q^2)_n}{(q^2;q^2)_n^2(q^{2s+2}/b, bq^{2s+2};q^2)_n} q^{2n}  \right. \\ &\left.  \qquad + ~ \Omega_2(a,b,q,m) \sum_{n=0}^{\frac{m-1}{2}-s} \frac{(q;q^2)_n^2(bq^{2s+1},q^{2s+1}/b;q^2)_n}{(q^2;q^2)_n^2(q^{2s+2}/a, aq^{2s+2};q^2)_n} q^{2n} \right\}.
\end{align*}
This completes the proof.
\end{proof}

\begin{proof} [Proof of Theorem \ref{P7 Thm 2}]
    Setting $b=1$ in $q$-congruence \eqref{P7 theorem 4 eqn}, we obtain
    \begin{align*} 
&\sum_{n=s}^{\frac{m-1}{2}}[4n+1] \frac{(q;q^2)_{n-s}(q;q^2)_{n+s} (aq,q/a;q^2)_n (q;q^2)_n^2}{(q^2;q^2)_{n-s}(q^2;q^2)_{n+s} (aq^2, q^2/a; q^2)_n (q^2;q^2)_n^2}q^n \notag \\
& \equiv [m+2s] \frac{(aq,q/a;q^2)_s (q;q^2)_s^2(q^2;q^2)_{\frac{m-1}{2}-s} (q;q^2)_{\frac{m-1}{2}+s}}{(aq^{2}, q^{2}/a;q^2)_s (q^2;q^2)_s^2(q;q^2)_{\frac{m-1}{2}-s} (q^2;q^2)_{\frac{m-1}{2}+s}} q^{\frac{1-m}{2}+2s} \notag \\
    & \quad \times \left\{ (1-q^m)^2 \sum_{n=0}^{\frac{m-1}{2}-s} \frac{(q;q^2)_n^2(aq^{2s+1},q^{2s+1}/a;q^2)_n}{(q^2;q^2)_n^2(q^{2s+2}, q^{2s+2};q^2)_n} q^{2n} \right. \notag \\ &\left.  \qquad \qquad + \frac{a(1-q^m)^2(2-q^m)}{(a-1)^2} \sum_{n=0}^{\frac{m-1}{2}-s} \frac{(q;q^2)_n^2(aq^{2s+1},q^{2s+1}/a;q^2)_n}{(q^2;q^2)_n^2(q^{2s+2}, q^{2s+2};q^2)_n} q^{2n}  \right. \notag \\ &\left.  \qquad \qquad - \frac{(1-aq^m)(a-q^m)(2-q^m)}{(a-1)^2} \sum_{n=0}^{\frac{m-1}{2}-s} \frac{(q;q^2)_n^2(q^{2s+1},q^{2s+1};q^2)_n}{(q^2;q^2)_n^2(q^{2s+2}/a, aq^{2s+2};q^2)_n} q^{2n} \right\},
\end{align*}
where we have used that $(a-1)^2 + a(2-q^m) = 1 + a^2 -aq^m.$ Applying L\rq H\^ospital\rq s rule, it is easy to see that
\begin{align*}
    & \lim_{a \to 1} \left\{\frac{a(1-q^m)^2}{(a-1)^2} \sum_{n=0}^{\frac{m-1}{2}-s} \frac{(q;q^2)_n^2(aq^{2s+1},q^{2s+1}/a;q^2)_n}{(q^2;q^2)_n^2(q^{2s+2}, q^{2s+2};q^2)_n} q^{2n} \right. \notag \\ &\left.  \qquad \qquad - \frac{(1-aq^m)(a-q^m)}{(a-1)^2} \sum_{n=0}^{\frac{m-1}{2}-s} \frac{(q;q^2)_n^2(q^{2s+1},q^{2s+1};q^2)_n}{(q^2;q^2)_n^2(q^{2s+2}/a, aq^{2s+2};q^2)_n} q^{2n} \right\} \\
    & = \sum_{n=0}^{\frac{m-1}{2}-s} \frac{(q;q^2)_n^2 (q^{2s+1};q^2)_n^2}{(q^2;q^2)_n^2(q^{2s+2};q^2)_n^2} q^{2n} \left\{ q^m - [m]^2 \sum_{i=1}^{n} \left( \frac{q^{2(s+i)-1}}{[2(s+i)-1]^2} + \frac{q^{2(s+i)}}{[2(s+i)]^2}\right)\right\}.
\end{align*}
    Since $(1-q^m)^2 + (2-q^m)q^m=1,$ the proof of Theorem \ref{P7 Thm 2} then follows by setting $a=1$.
\end{proof}

\section{Open problem} \label{Sec4}
We believe that the stronger version of Theorems \ref{P7 Thm 1} and \ref{P7 Thm 2}  should also be true. Based on the numerical evaluation, we propose the following conjectures as an open problem.
\begin{conjecture}
    The supercongruence \eqref{P7 Thm 1 eqn} is true modulo $[m]\Phi_m(q)^2$.
\end{conjecture}
\begin{conjecture}
    The supercongruence \eqref{P7 Thm 2 eqn} is true modulo $[m]\Phi_m(q)^4$.
\end{conjecture}
It is worth mentioning that the creative microscoping method used to investigate $q$-congruences does not seem to be applicable here due to the presence of the parameter $s$. It would be interesting to see if the reader can come to a conclusion regarding one of these conjectures.

\section*{Declarations}
	
{\bf Author contributions:} All authors have contributed equally to the preparation of this manuscript.\\
{\bf Funding:} No funding was received for this research.\\
{\bf Conflicts of Interest:} The authors do not have any conflicts
of interest. \\
{\bf Data Availability Statement:}
	No datasets were generated or analyzed in this study. The results are obtained entirely through theoretical and mathematical methods.

\end{document}